\documentclass[11pt]{amsart}

\usepackage{latexsym}
\usepackage{amsfonts}

\newcommand{\field}[1]{\mathbb{#1}}
\newcommand{\C}{\field{C}}

\newtheorem{defi}{Definition}[section]

\newtheorem{lem}[defi]{Lemma}
\newtheorem{theo}[defi]{Theorem}
\newtheorem{co}[defi]{Corollary}

\newtheorem{re}[defi]{Remark}

\font\tenmsy=msbm10

\def\Bbb#1{\hbox{\tenmsy#1}} 

\subjclass{14 A 10, 14 R 10, 51 M 99}

\title[On  finite regular and holomorphic mappings]{On  finite regular and holomorphic mappings} \makeatletter

\author{Zbigniew Jelonek}

\address[Z. Jelonek]{Instytut Matematyczny\\
Polska Akademia Nauk\\
\'Sniadeckich 8, 00-956 Warszawa, Poland}
\email{najelone@cyf-kr.edu.pl}
\thanks{The  author was partially supported by the NCN
 grant 2014-2017}

\date{\today}

\begin{document}

\maketitle

\begin{abstract}{Let $X, Y$ be smooth algebraic varieties of the same dimension.
Let $f, g : X \longrightarrow Y$ be finite regular mappings. We
say that $f, g$ are equivalent if there exists a regular
automorphism $\Phi \in \operatorname{Aut}(X)$ such that $f = g
\circ\Phi$. Of course if $f, g$ are equivalent, then they have the
same discriminants (i.e., the same set of critical values) and the
same geometric degree. We show that conversely, for every
hypersurface $V\subset Y$ and every $k\in \Bbb N,$ there are only
a finite number of non-equivalent finite regular mappings $f : X
\rightarrow Y$ such that the discriminant $D(f)$ equals $V$ and
$\mu(f ) = k$. As one of applications we show  that if  $f: X\to
Y$ is a finite mapping of topological degree two, then there
exists a regular automorphism $\Phi: X\to X$ which acts
transitively on the fibers of $f$ and $Y=X/G$, where $G=\{ id,
\Phi\}$ and $f$ is equivalent to the canonical projection $X\to
X/G.$

We prove the same statement in the local (and sometimes global)
holomorphic situation. In particular we show that if $f : (\mathbb
C^{n} , 0) \rightarrow (\mathbb C^{n} , 0)$ is a proper and
holomorphic mapping of topological degree two, then there exist
biholomorphisms $\Psi,\Phi : (\mathbb C^{n}  , 0) \rightarrow
(\mathbb C^{n} , 0)$ such that $\Psi\circ f\circ\Phi(x_{1} , x_{2}
, \dots, x_{n} ) = (x_{1}^{2} , x_{2} ,\dots, x_{n})$. Moreover,
for every proper holomorphic mapping $f : (\mathbb C^{n} , 0)
\rightarrow (\mathbb C^{n} , 0)$ with smooth discriminant there
exist biholomorphisms $\Psi,\Phi : (\mathbb C^{n} , 0) \rightarrow
(\mathbb C^{n}, 0)$ such that $\Psi\circ f\circ\Phi(x_{1} , x_{2}
, \dots, x_{n} ) = (x_{1}^{k} , x_{2} ,\dots, x_{n})$, where $k =
\mu(f )$. }
\end{abstract}

\section{Introduction}

\subsection{Global case}
Let $X,Y$ be smooth algebraic (resp. holomorphic) varieties of the
same dimension. We say that a regular (resp. holomorphic) mapping
$f: X\to Y$ is {\em finite}, if it is proper and it has finite
fibers. If $X,Y$ are affine (resp. Stein) then every proper
mapping is finite. Let $f, g:X\to Y$ be finite regular (resp.
holomorphic) mappings. By the {\em discriminant} $D(f)$ of $f$ we
mean the set of critical values of $f.$ Of course $D(f)$ is a
hypersurface in $Y.$ For every $y\in U=Y\setminus D(f)$ the fiber
$f^{-1}(y)$ has exactly $\mu(f)$ points, where $\mu(f)$ is the
topological degree of $f.$ In the case $X=Y=\C^2,$ S. Lamy
\cite{lam} (see also \cite{b-c} and \cite{b-c-1}) proved the
following theorem:

\begin{theo}\label{lamy} Let $f : \C^2\to  \C^2$ be a finite polynomial map of
topological  degree $2$. Then there exist polynomial automorphisms
$\Phi_1, \Phi_2$  such that $f = \Phi_2 \circ g \circ \Phi_1$,
where  $g(x, y) = (x, y^2).$
\end{theo}

This theorem, although  interesting, is very special. We explain
why it is so special and we generalize it in a few directions (see
Theorem \ref{lamhol}, Theorem \ref{lam} and Corollary \ref{wn2}).
In fact to describe a finite regular mapping it is not enough to
use only the geometric degree.  We show that in principle a finite
regular mapping is determined by its discriminant and the
geometric degree. The same  is true in a local (and sometimes in a
global) holomorphic case.

\begin{defi}
Let $X,Y$ be smooth algebraic (resp. holomorphic) manifolds of the
same dimension. Let $f, g: X\to Y$ be finite regular (resp.
holomorphic) mappings. We say that $f,g$ are equivalent if there
exists a regular automorphism (resp. a biholomorphism) $\Phi: X\to
X$ such that $f=g\circ \Phi.$ Moreover, we say that $f$ is weakly
equivalent to $g$ if there exist regular automorphisms (resp.
biholomorphisms) $\Phi\in Aut(X), \Psi\in Aut(Y)$ such that
$g=\Psi\circ f\circ \Phi.$
\end{defi}

Of course if $f,g$ are equivalent, then they have the same
discriminant and the same geometric degree. Our first main result
is the following:

\begin{theo}\label{gl}
Let $X,Y$ be smooth algebraic  varieties of the same dimension.
Let $V\subset Y$ be a hypersurface and $k$ be a natural number.
There are only a finite number of non-equivalent finite regular
mappings $f: X\to Y$ such that $D(f)=V$ and $\mu(f)=k.$
\end{theo}

\begin{re}\label{uwaga}
{\rm The same result is true in the holomorphic category, if we
additionally assume that the fundamental group of the space
$Y\setminus V$ is finitely generated.}
\end{re}

Note that in general there are two or more non-equivalent proper
mappings with the same discriminant and with the same geometric
degree. For example for $X=Y=\C^2$ there are exactly three
non-equivalent proper mappings  $f_1(x,y)=(x^4,y^4)$,
$f_2(x,y)=(x^8,y^2)$ and $f_3(x,y)=(x^2,y^8)$ with  $D(f)=\{
(x,y)\in \C^2 : xy=0\}$ and $\mu(f)=16$ (see Corollary \ref{wn2}).
However sometimes there is only one such a map:

\begin{co}\label{wn}
Let $V=\{ x\in \C^n : x_1=0\}.$ Every proper polynomial (resp.
holomorphic) mapping $f:\C^n \to\C^n$ with $D(f)=V$ and $\mu(f)=k$
is equivalent to the mapping $$g(x_1,x_2,\ldots, x_n)=(x_1^k,
x_2,\ldots, x_n).$$
\end{co}

This result has the following more general counterpart:

\begin{co}\label{wn1}
Let $V=\{ x\in \C^n : x_1=0\}$ and let $X$ be an $n$-dimensional
smooth affine (resp. Stein) variety. Then any two proper regular
(resp. holomorphic) mappings $f,g: X \to\C^n$ with $D(f)=V$ and
$\mu(f)=\mu(g)=k$ are  equivalent.  Moreover, there is then a
regular automorphism (resp. a biholomorphism) $\sigma: X\to X$ of
order $k$ such that $f\circ \sigma=f$ (i.e., $\sigma$ acts
transitively on the fibers of $f$).
\end{co}

 We have here one of possible  generalizations of Theorem \ref{lamy}
 (for other  generalizations see Theorems \ref{lamhol} and  \ref{lam}):

\begin{co}\label{wn2} Let $X,Y$ be smooth algebraic (resp. holomorphic)
manifolds of dimension $n.$ Let $f: X\to Y$ be a finite regular
(resp. holomorphic) mapping with $\mu(f)=2.$ Then there is an
automorphism (resp. a biholomorphism) $\sigma: X\to X$ of order
two such that $f\circ \sigma= f$ (i.e., $\sigma$ acts transitively
on fibers of $f$). The set of fixed points of $\sigma$ coincides
with the set of critical points of $f.$
\end{co}

As a  nice  application of Corollary \ref{wn2} we have:

\begin{co}\label{wniloraz}
If $f:X\to Z$ (where $X$ is smooth and $Z$ is normal) is a finite
surjective morphism (resp. holomorphic mapping) of topological
degree two, then there is a  subgroup $G\subset Aut(X)$ of order
two, such that

1) $Z\cong X/G,$

2) $f$ is equivalent to the projection $\pi: X\to X/G.$
\end{co}

We  also have:

\begin{co}\label{wn3}
Let $K_n=\{ x\in \C^n : \prod^n_{i=1} x_i=0\}.$ Every proper
polynomial (resp. holomorphic) mapping $f:\C^n \to\C^n$ with
$D(f)=V$ and $\mu(f)=k$ is equivalent to one of  the mappings
$$f_{d_1,\ldots,d_n}:\C^n\ni (x_1,\ldots, x_n)\mapsto (x_1^{d_1}, \ldots,
x_n^{d_n})\in \C^n,$$ where $\prod^n_{i=1} d_i=k$ and all $d_i>1.$
In particular if $k$ is a prime number and $n>1$, then there is no
proper polynomial (resp. holomorphic)  mappings with $D(f)=K_n$
and $\mu(f)=k.$
\end{co}

Our method can  also  sometimes be applied to arbitrary regular
mappings with finite fibers (i.e.,  quasi-finite mappings). Let us
recall that if $f:\C^n\to\C^n$ is a generically-finite regular
mapping, then  $B(f)=\{ x\in \C^n : \# f^{-1}(x)\not=\mu(f)\}$ is
called  the bifurcation set of $f.$ It is always a hypersurface
(see e.g. \cite{j-k}). We have:

\begin{co}\label{wn3'}
Let $K_n=\{ x\in \C^n : \prod^n_{i=1} x_i=0\}.$ Every quasi-finite
regular  mapping $f:\C^n \to\C^n$ with $B(f)\subset K_n$ is
proper.
\end{co}

If we use the concept of weak equivalence we have:

\begin{co}\label{wn4}
a) Let $f:\C\to \C$ be a  polynomial. If $D(f)$ is a  point, then
$f$ is weakly equivalent to the mapping $g(x)=x^k$, where
$k=\mu(f).$ Also there are  only a finite number of non weakly
equivalent polynomials $f$ of degree $k$ which have exactly two
critical points.

b) Let $f:\C^2\to \C^2$ be a proper polynomial mapping. If $D(f)$
is isomorphic to $\Bbb A^1(\C)$, then $f$ is weakly equivalent to
the mapping $g(x,y)=(x^k, y)$, where $k=\mu(f).$

c) Let $f:\C^n\to \C^n$ be a proper polynomial mapping. If $D(f)$
is isomorphic to $K_n=\{ x\in \C^n : \prod^n_{i=1} x_i=0\}$, then
$f$ is weakly equivalent to one of  the mappings
$$f_{d_1,\ldots,d_n}: \C^n\ni (x_1,\ldots, x_n)\mapsto (x_1^{d_1}, \ldots,
x_n^{d_n})\in\C^n,$$ where $\prod^n_{i=1} d_i=\mu(f)$ and all
$d_i>1.$
\end{co}

\subsection{Local case}
It is very interesting that our method works in the local
holomorphic setting as well. In particular we will prove:

\begin{theo}\label{glhol}
Let $(V,0)\subset (\C^n,0)$ be a germ of  analytic hypersurface
and let $k$ be a natural number. There are only a finite number of
non-equivalent proper holomorphic mappings $f: (\C^n,0)\to (\C^n,
0)$ such that $D(f)=V$ and $\mu(f)=k.$
\end{theo}

\begin{co}\label{wnhol}
Let $V=\{ x\in \C^n : x_1=0\}.$ Every proper  holomorphic mapping
$f:(\C^n,0) \to (\C^n,0)$ with $D(f)=(V,0)$ and $\mu(f)=k$ is
equivalent to the mapping $$g(x_1,x_2,\ldots, x_n)=(x_1^k,
x_2,\ldots, x_n).$$ In particular every proper holomorphic mapping
$f:(\C^n,0) \to (\C^n,0)$ with smooth discriminant  is weakly
equivalent to the mapping
$$g(x_1,x_2,\ldots, x_n)=(x_1^k, x_2,\ldots, x_n),$$
where $k=\mu(f).$
\end{co}

\begin{co}\label{wn2hol}
Let $K_n=\{ x\in \C^n : \prod^n_{i=1} x_i=0\}.$ Every proper
holomorphic mapping $f:(\C^n,0) \to (\C^n,0)$ with $D(f)=(K_n,0)$
and $\mu(f)=k$ is equivalent to one of  the mappings
$$f_{d_1,\ldots,d_n}:\C^n\ni (x_1, \ldots, x_n)\mapsto (x_1^{d_1}, \ldots,
x_n^{d_n})\in \C^n,$$ where $\prod^n_{i=1} d_i=k$ and all $d_i>1.$
In particular  every proper holomorphic mapping $f:(\C^n,0) \to
(\C^n,0)$ with $D(f)$ biholomorphic to $(K_n,0)$ and $\mu(f)=k$ is
weakly equivalent to one of these mappings.

Hence if $k$ is a prime number and $n>1$, then there is no proper
holomorphic  mapping $f:(\C^n,0) \to (\C^n,0)$ with $D(f)$
biholomorphic to $(K_n,0)$ and $\mu(f)=k.$
\end{co}

We  also have the following interesting:

\begin{theo}\label{lamhol}
Let $f:(\C^n,0) \to (\C^n, 0)$ be a proper  holomorphic mapping
with $\mu(f)=2.$ Then  $f$ is weakly equivalent to the mapping
$$g:(\C^n,0) \ni (x_1, x_2, \ldots, x_n)\mapsto (x_1^2, x_2,\ldots, x_n)\in (\C^n,0).$$
In particular the discriminant of $f$ is smooth.
\end{theo}

\vspace{5mm}  One can also prove an algebraic counterpart of
Theorem \ref{lamhol}. To do this we need:

\vspace{3mm}

\noindent {\bf Linearization Conjecture}. {\it Let $\Phi : \C^n\to
\C^n$ be a polynomial automorphism  of  order two. Then there
exists an automorphism $\Psi:\C^n\to \C^n$ such that
$$\Psi^{-1}\circ\Phi\circ \Psi \in GL(n).$$}
\noindent This Conjecture is true in dimension $2$ (and $n=1$ of
course) - see \cite{kam}, but unfortunately it is completely open
in higher dimensions. We have the following generalization of the
Lamy Theorem:

\begin{theo}\label{lam}
Let $f:\C^n \to\C^n$ be a proper polynomial  mapping with
$\mu(f)=2.$ If the Linearization Conjecture is true, then $f$ is
weakly equivalent to the mapping $$g:\C^n \ni (x_1, x_2, \ldots,
x_n)\mapsto (x_1^2, x_2,\ldots, x_n)\in\C^n.$$
\end{theo}

\noindent Sections 2 and 3  are devoted to the proofs of Theorem
\ref{gl} and Corollaries 1.4-1.11. In Section 4 we prove Theorem
\ref{glhol}  and Corollaries \ref{wnhol}-\ref{wn2hol}. In section
5 we prove Theorems \ref{lamhol} and  \ref{lam}.

\begin{re}\label{normal}
{\rm Our methods also work if we do not assume that $Y$ is smooth
and we define the discriminant of a finite mapping $f: X\to Y$ as
$D(f)=\{ y\in Y : y\in {\rm Sing}(Y) \ {\rm or}\ \# f^{-1}(y)<
\mu(f)\}.$}
\end{re}

\section{Proof of Theorem \ref{gl}}
Let $V\subset X$ be a hypersurface. It is well known that the
fundamental group of a smooth algebraic variety  is finitely
generated. In particular the group $\pi_1(X\setminus V)$ is
finitely generated. Let us recall the following result of M. Hall
(see  \cite{hal}, \cite{kur}):

\begin{lem}\label{hall}
Let $G$ be a finitely generated group and let $k$ be a natural
number. Then there are only finite number of subgroups $H\subset
G$ such that $[G:H]=k.$
\end{lem}

 \noindent Every finite
mapping $f: X\to Y$ with $D(f)=V$ and $\mu(f)=k$ induces a
topological covering $f:X\setminus f^{-1}(D(f))=P_f\to
R=Y\setminus D(f)$. Take a point $a\in R$ and let $a_f\in
f^{-1}(a).$ We have an induced homomorphism
$$f_* :\pi_1(P_f, a_f)\to \pi_1(R,a).$$
Denote $H_f=f_*(\pi_1(P_f, a_f))$ and $G=\pi_1(R,a).$ Hence
$[G:H_f]=k.$ By Lemma \ref{hall}  there are only a finite number
of subgroups $H_1,\ldots, H_r\subset G$ with index $k.$ Choose
finite regular mappings $f_i : X\to Y$ such that $H_{f_i}=H_i$ (of
course only if such a mapping $f_i$  does exist). We show that
every finite regular mapping $f$ such that $D(f)=V$ and $\mu(f)=k$
is equivalent to one of the mappings $f_i.$

Indeed, let $H_f=H_{f_i}.$ We show that $f$ is equivalent to
$f_i.$ Let us consider two coverings $f: (P_f, a_f)\to (R,a)$ and
$f_i: (P_{f_i}, a_{f_i})\to (R,a).$ Since $f_*(\pi_1(P_f,
a_f))={f_i}_*(\pi_1(P_{f_i}, a_{f_i}))$ we can lift the covering
$f$ to a homeomorphism $\phi: P_f\to P_{f_i}$ such that the
following diagram commutes:

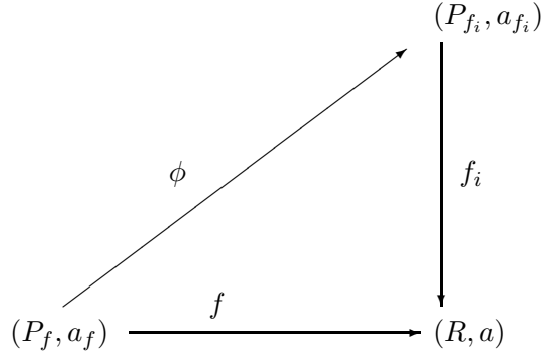
\begin{center}
\begin{picture}(240,120)(-40,40)
\put(180,160){\makebox(0,0)[tl]{$(P_{f_i}, a_{f_i})$}}
\put(20,40){\makebox(0,0)[tl]{$(P_f,a_f)$}}
\put(180,40){\makebox(0,0)[tl]{$(R,a)$}}
\put(190,100){\makebox(0,0)[tl]{$f_i$}}
\put(95,50){\makebox(0,0)[tl]{$f$}}
\put(80,100){\makebox(0,0)[tl]{$\phi$}}
\put(65,35){\vector(1,0){110}} \put(40,45){\vector(4,3){130}}
\put(183,145){\vector(0,-1){100}}
\end{picture}
\end{center}
\vspace{15mm}

\noindent Note that $\phi: P_f\to P_{f_i}$ is a single-valued
branch of the multi-valued holomorphic mapping $f_i^{-1}\circ f.$
In particular $\phi$ is a holomorphic mapping.

Now we show that $\phi$ has a unique extension to a regular
mapping on the whole of $X.$ Let $x\in X$ and $y=f(x).$ Let $U$ be
an affine neighborhood of $y$ in $Y$ and take $X_1=f^{-1}(U)$,
$X_2=f_i^{-1}(U).$ Since $f,f_i$ are finite, the varieties $X_1,
X_2$ are  affine. Moreover, $\phi(X_1)\subset X_2.$

Since $X_2$ is affine, we have an embedding $X_2\subset \C^N$ for
$N$ large enough. Consider the mapping  $\phi': P_f\cap X_1\ni
x\mapsto \phi(x)\in X_2 \subset \C^N.$ Since the mappings $f$ and
$f_i$ are proper the mapping $\phi'$ is locally bounded on $X_1.$
Hence by the Riemann Extension Theorem, $\phi'$ can be extended to
a holomorphic mapping $\Phi_{|X_1}: X_1\to \C^N$. Of course we
have $\Phi_{|X_1}(X_1)\subset X.$ Gluing all possible mappings
$\Phi_{|X_1}$ we get a global holomorphic extension $\Phi: X\to X$
of the mapping $\phi.$ Moreover we still have $f=f_i\circ \Phi.$

Finally, for fixed $X_1$ the graph of $\Phi_{|X_1}$ is a closed
irreducible analytic $n$-dimensional subset of the $n$-dimensional
affine variety $\Gamma=X_1\times_{Y} X_2$ (the fiber product given
by the mappings $f$ and $f_i$). Consequently, the ${\rm
graph}(\Phi_{|X_1})$ has to coincide with some irreducible
component of $\Gamma$, i.e., ${\rm graph}(\Phi)$ is an algebraic
subset of $X_1\times X_2.$ By the Serre Theorem about the
algebraic graph (see e.g., \cite{loj}, p. 342) we know that
$\Phi_{|X_1}$ is a regular mapping. Hence also  $\Phi$ is regular.
In a similar way the mapping $\Psi$ determined by $\phi^{-1}$ is
regular. It is easy to see that $\Psi\circ\Phi=\Phi\circ
\Psi=identity$, hence $\Phi$ is a regular automorphism.
Consequently, $f$ is equivalent to $f_i.$ $\square$

\begin{re}
{\rm It is easy to see from the proof that the number of
non-equivalent finite regular mappings $f: X\to Y$ such that
$D(f)=V$ and $\mu(f)=k$ is bounded by the number of subgroups of
index $k$ of the group $\pi_1(Y\setminus V).$}
\end{re}

\begin{re}
{\rm Assume that $X,Y$ are holomorphic and additionally the group
$\pi_1(Y\setminus V)$ is finitely generated. Then the holomorphic
counterpart of Theorem \ref{gl} is still true (this confirms
Remark \ref{uwaga}). The proof is the same as above, with one
exception: we have to prove in a different way that the mapping
$\phi$ has a holomorphic extension to the whole of $X.$ To do it
take a point $x\in f^{-1}(V)$ and let $y=f(x).$ The set
$f_i^{-1}(y)=\{ b_1,..., b_s\}$ is finite. Take small open
disjoint neighborhoods $W_i$ of $b_i$, such that each $W_i$ is
biholomorphic to a small ball in $\C^n$ (here $n={\rm dim} \ X$).
We can choose an
 open neighborhood $Q$ of $y$ so small that
$f_i^{-1}(Q)\subset \bigcup^s_{j=1} W_i.$  Now take a small
connected neighborhood $P_x$ of $x$ such that $P_x$ is
biholomorphic to a small ball in $\C^n$ and $f(P_x)\subset Q.$ The
set $P_x\setminus f^{-1}(V)$ is still connected and it is
transformed by $\phi$ into one particular set $W_{i_0}.$ Now we
can use the Riemann Extension Theorem.}
\end{re}

\section{Proof of Corollaries  1.4-1.11}

\noindent {\it Proof of Corollary \ref{wn}}. Throughout this
section, we keep notation of Section 2. Note that in this case
$G=\Bbb Z$ and for every $k$ there is only one subgroup of $G$
with index $k.$ In particular we can choose $f_1$ as
$$f_1(x_1,x_2,\ldots, x_n)=(x_1^k, x_2,\ldots, x_n).$$ Arguing as above
we see that every proper polynomial (holomorphic) mapping $f:\C^n
\to\C^n$ with $D(f)=V$ and $\mu(f)=k$ is equivalent to $f_1.$
$\square$

\vspace{5mm} \noindent {\it Proof of Corollary \ref{wn1}.} In this
case again we have $G=\Bbb Z.$ Let $f,g: X\to \C^n$ be proper
mappings with $D(f)=D(g)$ and $\mu(f)=\mu(g)=k.$ Since $G$ has
only one subgroup of index $k,$ we see as before that $f$ is
equivalent to $g.$ In the case of Stein manifolds use the fact,
that $R:=\C^n \setminus V$ is Stein, and since an analytic
covering of a Stein space is again Stein, we  also see  that
$P_f=X\setminus f^{-1}(V)$ and $P_g=X\setminus g^{-1}(V)$ are
Stein. Then one follows the proof of Theorem \ref{gl}.

Moreover, since the group $H_f$ is normal in $G,$ the covering $f:
P_f \to R$ is regular. In particular, the covering group
$G(P_f|R)$ acts transitively on $P_f.$ Since we know all covers of
the space $\C^n\setminus V=\C^*\times \C^{n-1}$ (it has the
homotopy type of a circle), we see that this group is cyclic. Now
take as $\sigma$ the generator of the group $G(P_f|R).$ As before,
$\sigma$ is a single-valued branch of the multi-valued holomorphic
function $f^{-1}\circ f,$  hence it is holomorphic. Again we can
extend $\sigma$ to the whole of $X$. Moreover, if $f$ is a regular
mapping, then (as above) so is $\sigma.$ $\square$

\vspace{5mm} \noindent {\it Proof of Corollary \ref{wn2}.} Let $X$
be an algebraic (resp. holomorphic) manifold. As above, the
mapping $f$ induces a topological regular covering $f:
P_f=X\setminus f^{-1}(V)\to R=Y\setminus V.$ For a point $a\in R$
let $a_1,a_2$ be two different points in the  fiber $f^{-1}(a).$
Since $f$ is a regular covering (the topological degree is two)
there exists a covering homeomorphism $\phi: P_f\to P_f$ such that
$\phi(a_1)=a_2.$ As before, the mapping $\phi$ is regular (resp.
holomorphic).

 Let
$x\in X$ be such that $y=f(x)\in V,$ in particular $f^{-1}(y)=x.$
Let $U\cong B(0,r)$ be a small neighborhood of $x$ in $X.$ Since
$f\circ \phi=f$ and $f$ is finite, there is a small neighborhood
$U'$ of $x$ such that $\phi(U'\setminus f^{-1}(V))\subset U.$ The
mapping $\phi: U'\setminus f^{-1}(V)\to U\cong B(0,r)\subset \C^n$
is bounded, hence by the Riemann Theorem it can be extended to a
holomorphic mapping $\Phi : U'\to U$ and $\Phi(x)=x$ for $x\in
f^{-1}(V)$. This implies  $Fix(\Phi)=C$ ( $C$ is the set of
critical points of $f$). $\square$

\vspace{5mm} \noindent {\it Proof of Corollary \ref{wniloraz}.} In
virtue of Remark \ref{normal} we have that there is a non-trivial
automorphism $\Phi: X\to X$ such that $f\circ\Phi=f.$ Take
$G=\{id, \Phi\}.$  Since the mapping $f$ is $G$-equivariant, we
have a well-defined morphism $f': X/G\ni [x]\mapsto f(x)\in Z.$ By
the assumption, the mapping $f'$ is a bijection. By the Zariski
Main Theorem (resp. by its holomorphic analogon), we have that
$f'$ is an isomorphism. Hence the mapping $f={f'}^{-1}\circ \pi$
is equivalent to the projection $\pi.$ $\square$

\begin{re}
{\rm We have used here the fact that every subgroup of $G$ of
index two is normal.  Note that this purely algebraic property is
one of the reasons why the case $\mu(f)=2$ is so special.}
\end{re}

\vspace{5mm} \noindent {\it Proof of Corollary \ref{wn3}}. Note
that $\C^n\setminus K_n\cong \prod^n_{i=1} \C^*$, hence $G=\Bbb Z
\oplus \ldots \oplus \Bbb Z$ ($n$ times). It is easy to see that
every subgroup of index $k$ in $G$, which comes from finite
mapping $f$  has the form $H(d_1,\ldots,d_n):=H_1\oplus \ldots
\oplus H_n,$ where $(\Bbb Z: H_i)=d_i$ and $\prod^n_{i=1} d_i=k.$
Indeed, take $F_i=\{ x\in \C^n: f_i=0 \}$ and  $W_i=\{ x\in \C^n :
x_i=0\}.$ Consider  groups $H_1(\C^n\setminus \bigcup^n_{i=1}
W_i)\cong \pi_1(\C^n\setminus K_n)= \Bbb Z^n$ and 
$H_1(\C^n\setminus \bigcup^n_{i=1} F_i)\cong
f_*(\pi_1(\C^n\setminus \bigcup^n_{i=1} F_i)\subset G= \Bbb Z^n.$
This implies that $H_1(\C^n\setminus \bigcup^n_{i=1} F_i)=\Bbb
Z^n.$ By the Mayer-Vietoris sequence we have
$$H_1(\C^n\setminus \bigcup^n_{i=1} F_i)=\bigoplus^n_{i=1}
H_1(\C^n\setminus F_i)$$ and $$H_1(\C^n\setminus
K_n)=\bigoplus^n_{i=1} H_1(\C^n\setminus W_i).$$ The induced
mapping $H_1(f): \bigoplus^n_{i=1} H_1(\C^n\setminus F_i)\to
\bigoplus^n_{i=1} H_1(\C^n\setminus W_i)$ coincide with the
mapping $H_1(f): H_1(\C^n\setminus \bigcup^n_{i=1} F_i)\to
H_1(\C^n\setminus K_n)$ and shows that indeed the group
$H(d_1,..., d_n)$ has a desired form.

The mapping $f_{d_1,\ldots,d_n} : \C^n \ni ( x_1,\ldots,
x_n)\mapsto (x_1^{d_1},\ldots, x_n^{d_n})\in \C^n$, restricted to
the set $\C^n\setminus K_n$, satisfies
$${f_{d_1,\ldots,d_n}}_*(\pi_1(\C^n\setminus K_n, a_f))=H(d_1,\ldots, d_n).$$
This implies that every proper polynomial (resp. holomorphic)
mapping with $D(f)=K_n$ and $\mu(f)=k$ has to be equivalent to one
of the mappings $f_{d_1,\ldots, d_n}.$  However, we have to
exclude mappings for which some of $d_i$ is equal to one, because
in this case $D(f)$ does not contain the hyperplane $x_i=0.$
$\square$

\vspace{5mm} \noindent {\it Proof of Corollary \ref{wn3'}}.   The
mapping $f$ induces a topological covering $f: P_f=\C^n\setminus
f^{-1}(K_n)\to R=\C^n\setminus K_n$ of degree $\mu(f).$ Choose a
point $a\in R$ and a point $a_f\in f^{-1}(a).$ Let $H(d_1,\ldots,
d_n)$ be the group $f_*(\pi_1(P_f, a_f))\subset \pi_1(R, a).$

As in the proof of Corollary \ref{wn3} there is a mapping
$f_{d_1,\ldots,d_n} : \C^n \ni ( x_1,\ldots, x_n)\mapsto
(x_1^{d_1},\ldots, x_n^{d_n})\in \C^n$ (and a point $b\in
f_{d_1,\ldots,d_n}^{-1}(a)$), which when restricted to the set
$\C^n\setminus K_n$, satisfies
${f_{d_1,\ldots,d_n}}_*(\pi_1(\C^n\setminus K_n, b))=H(d_1,\ldots,
d_n).$

As before, we have a biholomorphic mapping $\phi: P_f \to
P_{f_{d_1,\ldots,d_n}}$ such that $f=f_{d_1,\ldots,d_n}\circ
\phi.$ Since the mapping $f_{d_1,\ldots,d_n}$ is proper, we can
extend  $\phi$ to a unique birational mapping $\Phi: \C^n\to \C^n$
and still  have $f=f_{d_1,\ldots,d_n}\circ \Phi.$ However, the
mapping $f$ is quasi-finite, hence so is $\Phi.$ Now by the
Zariski Main Theorem the mapping $\Phi$ is a regular automorphism.
$\square$

\vspace{5mm} \noindent {\it Proof of Corollary \ref{wn4}}. In
dimension one the hypersurface $X\subset \C$ which contains one or
two points has only one  embedding in $\C$ (up to equivalence). In
dimension two by the famous Abhyankar-Moh-Suzuki Theorem the line
$\Bbb A^1(\C)$ has only one embedding in $\C^2.$ Finally the
$n-$cross $K_n=\{ x\in \C^n : \prod^n_{i=1} x_i=0\}$ has only one
 embedding into $\C^n$ (see \cite{jel}). By the remarks above, we can
choose the polynomial automorphism $\Psi: \C^n\to\C^n$ such that
$D(\Psi\circ f)$ is equal to:

a)   the point $0$ or  the set $\{0,1\}$ respectively,

b)  the line $\{ x=0\}\subset \C^2$,

c)  the $n$-cross $K_n\subset \C^n.$

 \noindent Now it is enough to apply Corollary \ref{wn} or Corollary   \ref{wn3}.$\square$

\section{Proofs of Theorem \ref{glhol} and Corollaries
\ref{wnhol}-\ref{wn2hol}}

\noindent{\it Proof of Theorem \ref{glhol}.} Let $B(0,\rho)$ be a
small ball around the origin.  It is well known that the
fundamental group of the space $U_\rho=B(0,\rho)\setminus V$ does
not depend on sufficiently small $\rho$, and this group is
finitely generated.  Moreover, it is easy to see that for small
$\rho_1<\rho$ the space $U_{\rho_1}$ is a deformation retract of
$U_\rho.$

Every proper mapping $f: (\C^n,0)\to (\C^n,0)$ with $D(f)=(V,0)$
and $\mu(f)=k$ induces a topological covering $f:
f^{-1}(U_\rho)=P_{f,\rho}\to U_\rho$. Take a point $a\in U_\rho$
and let $a_f\in f^{-1}(a).$ We have an induced homomorphism
$$f_* :\pi_1(P_{f,\rho}, a_f)\to \pi_1(U_\rho,a).$$
Denote $H_{f,\rho}=f_*(\pi_1(P_{f,\rho}, a_f))$ and
$G=\pi_1(U_\rho,a).$ Note that for small $\rho_1<\rho$ the space
$P_{f,\rho_1}$ is a deformation retract of $P_{f,\rho}.$ In
particular the group $H_{f,\rho}$ does not depend on $\rho$; we
will denote it by $H_f.$ We also write $P_f$ instead of
$P_{f,\rho},$ and $U$ instead of $U_\rho.$

Hence $[G:H_f]=k.$ By Lemma \ref{hall}  there are only a finite
number subgroups $H_1,\ldots, H_r\subset G$ with index $k.$ Choose
proper holomorphic mappings $f_i : (\C^n,0)\to (\C^n,0)$ such that
$H_{f_i}=H_i$ (of course only if such a $f_i$   exists). We show
that every proper holomorphic mapping $f: (\C^n,0)\to (\C^n,0)$
such that $D(f)=(V,0)$ and $\mu(f)=k$ is equivalent to one of the
$f_i.$

Indeed, let $H_f=H_{f_i}.$ We show that $f$ is equivalent to
$f_i.$ Let us consider two coverings $f: (P_f, a_f)\to (U,a)$ and
$f_i: (P_{f_i}, a_{f_i})\to (U,a).$ Since $f_*(\pi_1(P_f,
a_f))={f_i}_*(\pi_1(P_{f_i}, a_{f_i}))$ we can lift the covering
$f$ to a homeomorphism $\phi: P_f\to P_{f_i}$ such that the
following diagram commutes:
\begin{center}
\begin{picture}(240,120)(-40,40)
\put(180,160){\makebox(0,0)[tl]{$(P_{f_i}, a_{f_i})$}}
\put(20,40){\makebox(0,0)[tl]{$(P_f,a_f)$}}
\put(180,40){\makebox(0,0)[tl]{$(U,a)$}}
\put(190,100){\makebox(0,0)[tl]{$f_i$}}
\put(95,50){\makebox(0,0)[tl]{$f$}}
\put(80,100){\makebox(0,0)[tl]{$\phi$}}
\put(65,35){\vector(1,0){110}} \put(40,45){\vector(4,3){130}}
\put(183,145){\vector(0,-1){100}}
\end{picture}
\end{center}
\vspace{15mm}

\noindent Note that $\phi: P_f\to P_{f_i}$ is a single-valued
branch of the multi-valued holomorphic mapping $f_i^{-1}\circ f.$
In particular $\phi$ is a holomorphic mapping. Further, the
mapping $\phi$ is  bounded on $P_f.$ Hence by the Riemann
Extension Theorem $\phi$ can be extended to a holomorphic mapping
$\Phi: f^{-1}(B(0,\rho))\to f_i^{-1}(B(0,\rho))$. Moreover  we
still have $f=f_i\circ \Phi.$

In a similar way the mapping $\Psi: f_i^{-1}(B(0,\rho))\to
f^{-1}(B(0,\rho))$ determined by $\phi^{-1}$ is holomorphic. It is
easy to see that $\Psi=\Phi^{-1}$, hence $\Phi$ is a
biholomorphism. Consequently, $f$ is equivalent to $f_i.$
$\square$

\vspace{5mm} \noindent  {\it Proof of Corollary \ref{wnhol}.} Let
$U_\rho=B(0,\rho)\setminus V.$
 It is easy to see that $U_\rho$ is a deformation retract of
 $\C^n\setminus V.$ In particular  $G=\Bbb Z.$
 Now the rest of the proof of the first part
 of Corollary \ref{wnhol} is the same as in the proof of Corollary
 \ref{wn} so we skip it.

Now assume that  $f:(\C^n,0) \to (\C^n,0)$ is a proper holomorphic
mapping with a smooth discriminant $D(f).$ Let $\Psi:(\C^n,0)\to
(\C^n,0)$ be a biholomorphism such that $\Psi(D(f))=(V,0),$ where
$V=\{x: x_1=0\}.$ Hence $D(\Psi\circ f)=(V,0)$ and by the first
part $\Psi\circ f$ is equivalent to $g$.

\vspace{5mm}

\noindent {\it Proof of Corollary \ref{wn2hol}.} Again  $U_\rho$
is a deformation retract of
 $\C^n\setminus K_n.$ In particular $G=\Bbb
Z \oplus \ldots \oplus \Bbb Z$ ($n$ times). Now the rest of the
proof of the first part
 of Corollary \ref{wn2hol} is the same as in the proof of Corollary
 \ref{wn3} so we skip it.

Assume now that $D(f)$ is biholomorphic to $K_n$, and let $\sigma:
(K_n,0)\to (D(f),0)$ be this isomorphism. Let $\Lambda_i=\sigma(\{
x: x_i=0\})$ and let $h_i=0$ be a local reduced equation of
$\Lambda_i.$ Take $\Psi=(h_1,\ldots, h_n).$ Then $\Psi$ is a local
biholomorphism (all $\Lambda_i$ intersect transversally at $0$ !)
and it is easy to see that $D(\Psi\circ f)=(K_n,0).$ Now we can
apply the first part of the Corollary. $\square$

\section{Proofs of Theorems \ref{lamhol} and  \ref{lam}}

\noindent {\it Proof of Theorem \ref{lamhol}.} Let $f:(\C^n,0)\to
(\C^n,0)$ be a proper holomorphic mapping of degree two.  Let
$B(0,\rho)$ be a small ball around the origin and
$U_\rho=B(0,\rho)\setminus V.$ Take a point $a\in U_\rho$ and let
$b_1, b_2\in f^{-1}(a).$ We have the induced homomorphisms
$$f_* :\pi_1(P_{f}, b_i)\to \pi_1(U_\rho,a), \ i=1,2.$$
Denote $H_i=f_*(\pi_1(P_{f}, b_i))$ for $i=1,2$ and denote as
above $G=\pi_1(U_\rho,a).$

Since $[G:H_i]=2$ we see that each $H_i$ is a are normal subgroup
of $G.$ From topology we know that $H_1$ is conjugate to $H_2$,
and consequently $H_1=H_2.$ Let us consider two coverings $f:(P_f,
b_1)\to (U,a)$ and $f:(P_{f}, b_2)\to (U,a).$ Since
$f_*(\pi_1(P_f, b_1))={f}_*(\pi_1(P_{f}, b_2))$ we can lift the
former covering  to a homeomorphism $\phi: (P_f, b_1)\to (P_{f},
b_2)$ such that the following diagram commutes:

\begin{center}
\begin{picture}(240,120)(-40,40)
\put(180,160){\makebox(0,0)[tl]{$(P_{f}, b_2)$}}
\put(20,40){\makebox(0,0)[tl]{$(P_f, b_1)$}}
\put(180,40){\makebox(0,0)[tl]{$(U,a)$}}
\put(190,100){\makebox(0,0)[tl]{$f$}}
\put(95,50){\makebox(0,0)[tl]{$f$}}
\put(80,100){\makebox(0,0)[tl]{$\phi$}}
\put(65,35){\vector(1,0){110}} \put(40,45){\vector(4,3){130}}
\put(183,145){\vector(0,-1){100}}
\end{picture}
\end{center}
\vspace{15mm}

\noindent Put $R=f^{-1}(B(0,\rho)).$ As above,  $\phi$ can be
extended to a biholomorphism $\Phi: R\to R$ of order two with
$f=f\circ \Phi.$ Since $\Phi$ permutes the fibers of $f$, we see
that all critical points of $f$ are fixed points of $\Phi.$ In
particular $\Phi(0)=0.$ Now by the Cartan Theorem (see \cite{car})
we know that for $\rho$ small enough, there is a  biholomorphism
$\Sigma: (R,0)\to (R,0)$ such that $\Sigma^{-1}\circ \Phi\circ
\Sigma:=L$ is a linear automorphism. Now we need the following:

\begin{lem}\label{lemat}
Let $L:\C^n\to \C^n$ be a linear mapping of order two and assume
that the set of fixed points of $L$ is a hyperplane $W$. Then in
some coordinates, $$L(x_1, x_2, \ldots, x_n)=(-x_1, x_2,\ldots,
x_n).$$
\end{lem}

\begin{proof}
Take a basis $e_1,\ldots, e_n$ in $\C^n$ such that $e_2,\ldots,
e_{n}$ span the hyperplane $W.$ Hence $L(e_i)=e_i$ for $i>1$ and
$L(e_1)=\sum^n_{i=1} a_i e_i.$ In particular $L(x_1 e_1+\ldots+x_n
e_n)=(a_1x_1)e_1+(x_2+a_2x_1)e_2+ \ldots+(x_{n}+a_{n}x_1)e_{n}.$
Since $L$ has  finite order, we have $a_i=0$ for $i>1.$ In
particular $L(x_1,\ldots, x_n)=(a_1 x_1, x_2,\ldots, x_{n}).$
However det $L=\pm 1$, i.e., $a_1=\pm 1.$ Of course $a_1\not=1.$
\end{proof}

Let $L$ be a linear automorphism as in Lemma \ref{lemat} and put
$G=(id, L).$ The group $G$ acts on $\C^n$ and we have
$\C[x_1,\ldots, x_n]^G=\C[x_1^2, x_2,\ldots, x_n].$ In particular
$$\C^n/G=Spec(\C[x_1,\ldots, x_n]^G)=Spec(\C[x_1^2, x_2,\ldots,
x_n])=\C^n.$$ Under this identification the canonical mapping
$\pi: \C^n\to \C^n/G$ coincides with the mapping $\C^n\ni
(x_1,\ldots, x_n)\mapsto (x_1^2, x_2,\ldots, x_n)\in \C^n.$

Take $f'=f\circ \Sigma.$ Hence $f'$ is equivalent to $f$ and
$f'\circ L=f'.$ The mapping $f'$ induces a mapping $\tilde{f}:
R/G=\pi(R)\to B(0,\rho),$ where $f'=\tilde{f}\circ \pi.$ Since
$\tilde{f}$ has degree one it is a biholomorphism. However
$\tilde{f}^{-1}\circ f'=\pi$, so the mapping $f'$ (hence also
$f$) is weakly equivalent to
$$\pi :(\C^n,0) \ni (x_1,\ldots, x_n)\mapsto (x_1^2, x_2,\ldots, x_n)\in (\C^n,0). \square$$

\vspace{5mm}

\noindent {\it Proof of Theorem \ref{lam}.}  By Corollary
\ref{wn2} there exists a non-trivial automorphism $\Phi :
\C^n\to\C^n$ such that $f\circ  \Phi=f.$ In particular the set of
fixed points of $\Phi$ coincide with the set of critical points of
$f.$ If the Linearization Conjecture is true, then
$\Sigma^{-1}\circ \Phi \circ \Sigma=L\in GL(n)$ for some
automorphism $\Sigma\in Aut(\C^n).$ As above we can assume that
$L(x_1, x_2, \ldots, x_n)=(-x_1, x_2, \ldots, x_n)$. Put
$f:=f\circ \Sigma.$ Now we have $f\circ L=f.$

If $G=(id, L),$ then we have $\C^n/G=\C^n$ and the canonical
mapping $\pi: \C^n\to \C^n/G$ coincides with the mapping $\C^n\ni
(x_1,\ldots, x_n)\mapsto (x_1^2, x_2,\ldots, x_n)\in \C^n.$ Since
$f$ is $G$-invariant, the mapping
$$\tilde{f}=f\circ \pi^{-1} : \C^n/G\ni (x_1,\ldots, x_n)\mapsto
f(\sqrt{x_1}, x_2,\ldots, x_n)\in \C^n$$ is well defined and
polynomial. It is easy to check that $\tilde{f}$ has degree one,
i.e., it is an automorphism. Finally $\tilde{f}^{-1}\circ f= \pi.$
$\square$

\vspace{3mm}

In particular for $n=2$ we get another proof of the Lamy Theorem
(Theorem \ref{lamy}).

\end{document}